\documentclass[12pt]{amsart}
\usepackage[lite]{amsrefs}

\usepackage{amssymb}
\usepackage[T1]{fontenc}
\usepackage{mathrsfs}  
\usepackage[all,cmtip]{xy}

\usepackage {amsfonts, amssymb, amscd,amsthm, amsmath, eucal}

\usepackage{amsbsy}
\usepackage[all]{xy}

\usepackage{color}

\theoremstyle{plain} {
  \newtheorem{thm}{Theorem}[section]
  \newtheorem{defn}[thm]{Definition}
  \newtheorem{cor}[thm]{Corollary}
  \newtheorem{lem}[thm]{Lemma}
  \newtheorem{prop}[thm]{Proposition}
  \theoremstyle{definition}
  \newtheorem{rem}[thm]{Remark}
    \newtheorem{constr}[thm]{Construction}
  \theoremstyle{plain}
  \newtheorem{clm}[thm]{Claim}
  \newtheorem{notation}[thm]{Notation}
  
  \newtheorem{sssection}[thm]{}
}

{
}
\renewcommand{\subsubsection}{\sssection\rm}

\newcommand{\bE}{\mathbf E}
\newcommand{\bG}{\mathbf G}
\newcommand{\bH}{\mathbf H}

\newcommand{\bP}{\mathbf P}
\newcommand{\bT}{\mathbf T}
\newcommand{\bU}{\mathbf U}
\newcommand{\bu}{\mathbf u}
\newcommand{\bB}{\mathbf B}

\newcommand{\cA}{\mathcal A}
\newcommand{\cB}{\mathcal B}
\newcommand{\cE}{\mathcal E}
\newcommand{\cF}{\mathcal F}
\newcommand{\cG}{\mathcal G}

\newcommand{\cO}{\mathcal O}

\newcommand{\cT}{\mathcal T}

\renewcommand{\P}{\mathbb P}
\DeclareMathOperator{\Gl}{Gl}

\DeclareMathOperator{\spec}{Spec}

\newcommand{\can}{\text{\rm can}}

\newcommand{\id}{\text{\rm id}}
\newcommand{\pr}{\text{\rm pr}}

\newcommand{\inc}{\text{\rm inc}}

\newcommand{\const}{\text{\rm const}}
\newcommand{\Spec}{\text{\rm Spec}}

\newcommand{\T}{\mathcal {T}}
\newcommand{\Aff}{\mathbf {A}}
\newcommand{\Pro}{\mathbf {P}}

\newcommand \xra {\xrightarrow }

\newcommand \hra {\hookrightarrow }

\newcommand{\ttf}{{\text{f}}}
\newcommand{\E}{\mathcal E}

\renewcommand{\P}{\mathbb P}

\newcommand\mydim{\text{\rm dim}}

\renewcommand \id{\operatorname{id}}

\renewcommand \phi\varphi

\newcommand{\et}{\text{\rm\'et}}
\newcommand{\R}{{\rm R}}

\def \F{\mathcal F}

\newcommand{\ZZ}{\mathbb Z}

\begin{document}

\title{On the Grothendieck--Serre conjecture for projective smooth schemes over a DVR
%on principal bundles over regular local rings containing a field
}

\author{Ning Guo}
\author{Ivan Panin}
\address{St. Petersburg branch of V. A. Steklov Mathematical Institute, Fontanka 27, 191023 St. Petersburg, Russia}
\email{guo.ning@eimi.ru}
\email{paniniv@gmail.com}
\date{\today}

%\c{C}
\maketitle

\begin{abstract}
The Grothendieck--Serre conjecture predicts that every generically trivial torsor under a reductive group scheme $\bG$
over a regular local ring $R$ is trivial. The mixed characteristic case of the conjecture is widely open. We consider the following set up.
Let $A$ be a mixed characteristic DVR, $\bG$ a reductive group scheme over $A$, $X$ an irreducible smooth projective $A$-scheme,
$\cG$ a principal $\bG$-bundle over $X$. Suppose $\cG$ is generically trivial. We prove that in this case $\cG$ is Zariski locally trivial.
This result confirms the conjecture.
\end{abstract}

\section{Main results}\label{Introduction}
A survey paper \cite{P} on this
topic was published in the proceedings of ICM-2018.
Also good references for the history of the topic are
\cite{FP} and \cite{C}.
Let $R$ be a commutative unital ring. Recall that an $R$-group scheme $\bG$ is called \emph{reductive},
if it is affine and smooth as an $R$-scheme and if, moreover,
for each algebraically closed field $\Omega$ and for each ring homomorphism $R\to\Omega$ the scalar extension $\bG_\Omega$ is
a connected reductive algebraic group over $\Omega$. This definition of a reductive $R$-group scheme
coincides with~\cite[Exp.~XIX, Definition~2.7]{SGA3}.

Assume that $X$ is a regular scheme, $\bG$ is a reductive $X$-group scheme.
Recall that a $X$-scheme $\cG$ with an action of $\bG$ is called \emph{a principal $\bG$-bundle over $X$},
if $\cG$ is faithfully flat and quasi-compact over $X$ and the action is simply transitive,
that is, the natural morphism $\bG\times_X\cG\to\cG\times_X\cG$ is an isomorphism, see~\cite[Section~6]{Gr4}.
It is well known that such a bundle is trivial locally in \'etale topology but in general not in the Zariski topology.
Grothendieck and Serre conjectured that $\cG$ is trivial locally in the Zariski topology, if it is trivial generically.

More precisely,
a~well-known conjecture due to J.-P.~Serre and A.~Grothendieck
(see~\cite[Remarque, p.31]{Se}, \cite[Remarque 3, p.26-27]{Gr1}, and~\cite[Remarque~1.11.a]{Gr2})
asserts that given a regular local ring $R$ and its field of fractions~$K$ and given a reductive group scheme $\bG$ over $R$, the map
\[
  H^1_{\text{\'et}}(R,\bG)\to H^1_{\text{\'et}}(K,\bG),
\]
induced by the inclusion of $R$ into $K$, has a trivial kernel.
A {\it survey paper} \cite{P} on the topic is published in proceedings of ICM-2018. \\\\
If $R$ contains a field, then the conjecture is proved in \cite{Pan3}.
If $R$ contains {\it an infinite} field, then the conjecture is proved earlier in \cite{FP}. \\\\
The mixed characteristic case of the conjecture is widely open. Nevertheless there are several
rather general results proven in \cite{NG}, \cite{C}, \cite{Tsy}. Point out that
the main theorems of \cite{NG} and \cite{Tsy} are used in this preprint. \\\\
{\bf The main aim of this preprint is to prove Theorem \ref{MainThm1}}.\\
$A$ is a d.v.r., $m_A\subseteq A$ is its maximal ideal;\\
%$\pi \in m_A$ is a generator of the maximal ideal; \\
$K$ is the field of fractions of $A$;
$V=Spec(A)$ and $v\in V$ is its closed point;\\
$k(v)$ is the residue field $A/m_A$;\\
$p>0$ is the characteristic of the field $k(v)$; \\
it is supposed in this preprint that the field $K$ has characteritic zero;\\
%%%$m\geq 1$ is an integer;\\
For an $A$-scheme $Y$ we write $Y_v$ for the closed fibre of the structure morphism $Y\to V$.\\
For an $A$-scheme morphism $q: Y'\to Y$ write $q_v: Y'_v\to Y_v$ for the morphism $q\times_V v$;\\
All schemes in this section are $V$-schemes and we write $Y\times Y'$ for $Y\times_V Y'$. \\\\
%Let $S_A = A[x_0,...,x_r]$ be the homogeneous coordinate ring of $\mathbb P^r_V$,
%let $S_{A,d} \subset S_A$ be the $A$-submodule of homogeneous polynomials of degree $d$,
%and let $S_{A,hom} = \sqcup^{\infty}_{d=0} S_{A,d}$ (the disjoint union).
%Let $X\subset \mathbb P^r_V$ be a closed subscheme, which is $V$-smooth, irreducible of relative dimension $m$.\\
%We will write $\mathcal O(e)$ for $\mathcal O_{\mathbb P^r_V}(e)|_X$; \\
%Let $Z\subset X$ be a closed subset such that $dim Z_v\leq m-1$.\\
%For each $F\in S_{A,hom}$ let $X_F$ be the open subscheme $\{F\neq 0\}$ in $X$; \\
%Let $f_0,...f_m\in S_{hom}$ be as in Theorem \ref{Main3} with respect to subscheme $X_v\subset \mathbb P^r_v$ and its closed subset $Z_v\subset X_v$. \\
%Let $e_0,...,e_m$ be the degrees of $f_0,...f_m$ respectively.\\
%Let $F_0,...,F_m\in S_{A,hom}$ be such that $F_i \ mod \ m\equiv f_i$.\\
%Write $s_i\in \Gamma(X,\mathcal O(e_i))$ for $F_i|_X$.\\
%Let $q=(F_2/F^{e_2/e_1}_1,...,F_m/F^{e_m/e_1}_1): X_{F_1}\to \mathbb A^{m-1}_V$;\\
%For a Zariski open $S$ in $\mathbb A^{m-1}_V$ put $X_S=q^{-1}(S)$ and $Z_S=Z\cap X_S$.
%%%$d\geq 1$ is an integer;
Let $X$ be an irreducible $V$-scheme which is $V$-smooth and $V$-projective \\
of relative dimension $d\geq 1$ over $V$.  \\
Let $\mathcal K$ be the rational function field of $X$ and $\eta=Spec \mathcal \ \mathcal K$ be the generic point of $X$. \\
Let $\bG$ be a reductive group scheme over $A$. \\
{\bf Let $\cO$ be the semi-local ring of finitely many closed points of $X$}.\\
Put $U:=Spec \ \cO$.

\begin{thm}\label{MainThm1}
Let $A$, $\bG$, $X$, $\eta \in X$ be as above.
%$R$ be a regular semi-local domain containing a field, and let $K$ be its field of fractions.
Then the kernel of the pull-back map
\[
  H^1_{\text{\'et}}(X,\bG)\to H^1_{\text{\'et}}(\eta,\bG),
\]
\noindent
consists of principal $\bG$-bundle over $X$ which are Zariski semi-locally trivial.
%induced by the inclusion of $\mathcal O$ into $\mathcal K$, has a trivial kernel.
%In other words, under the above assumptions on $\mathcal O$ and $\bG$, each principal $\bG$-bundle over $R$ having a $K$-rational point is trivial.
\end{thm}

%%%%%Note that, while Theorem~\ref{MainThm1} was previously known for reductive group schemes $\bG$ coming from the ground field
%%%%%({\it an unpublished result due to O.Gabber}), in many cases the corollary is a new result even for such group schemes.
%Let us illustrate this comment as follows.
%{\bf It is well known that such a bundle is trivial locally in the \'{e}tale topology but in general
%not in the Zariski topology}
%\begin{cor}\label{two_bundles}
%Let $k$ be a field, $\bG$ be reductive group scheme over $k$. Let $X$ be an affine irreducible $k$-smooth variety
%and $\cG_1$, $\cG_2$ be two principal $\bG$-bundles over $X$. Suppose $f\in k[X]$ is a non-zero function such that
%principal $\bG$-bundles $\cG_1$, $\cG_2$ are isomorphic over $X_f$. Then principal $\bG$-bundles $\cG_1$, $\cG_2$ are isomorphic
%locally for the Zariski topology on $X$.
%\end{cor}

For a $V$-scheme $W$ we denote by $\mathbb A^1_W$ the affine line over $W$ and by $\P^1_W$ the projective line over $W$.
Let $T$ be a $W$-scheme. By a principal $\bG$-bundle over $T$ we understand a principal $\bG\times_W T$-bundle.
We refer to
\cite[Exp.~XXIV, Sect.~5.3]{SGA3}
for the definitions of
a simple simply-connected group scheme over a scheme
and a semi-simple simply-connected group scheme over a scheme.
%(see~\cite[Exp.~XXIV, Sect.~5.3]{SGA3} for the definition).

In Section~\ref{sect:outline}
we derive Theorem~\ref{MainThm1} from the following two results.
%the first two of them are of independent interest

%In Section~\ref{Reducing_MainThm1_to_two_others}
%we deduce Theorem~\ref{MainThm1} from the following two results of independent interest

%~\ref{sect:redtopsv}
%(cf.~\cite[Thm.1.3]{PSV}).
\begin{thm}({\rm{[GPan]})}
\label{MainHomotopy}
Let $A$, $\bG$, $X$, $\eta\in X$, $U$ be as in Theorem \ref{MainThm1}.
%Let $\mathcal O$ be the local ring of a closed point on $X$.
%and let $K$ be its
%field of fractions.
%Set $U=\spec \mathcal O$.
%Let $G$ be a simple simply connected
%group scheme over $\mathcal O$.
%Let $k$, $\mathcal O$, $K$ be the same as
%in Theorem \ref{MainThmGeometric}.
%%%Let $\bG$ be {\bf a simple simply-connected} group scheme over $U$.
%Let $\bG$ be a reductive
%group scheme over $\mathcal O$.
%%%$\mathcal O$.
Let $\mathcal G$ be a principal $\bG$-bundle over $X$
%$\mathcal O$
trivial over the generic point $\eta$ of $X$.
%$K$.
Then there
exists a principal $\bG$-bundle $\mathcal G_t$ over $\mathbb P^1_U$
%$\mathcal O[t]$
and a closed subset $\mathcal Z$ in $\mathbb A^1_U$ finite over $U$
% monic polynomial $h(t) \in \mathcal O[t]$
such that
\par
(i) the $\bG$-bundle $\mathcal G_t$ is trivial over the open subscheme
$\mathbb P^1_U-\mathcal Z$ in $\mathbb P^1_U$,
% where $Z=\{h(t)=0\}\subset \mathbb A^1_U$ is closed as in
%$\mathbb A^1_U$, so in $\mathbb P^1_U$;
%$\mathcal O[t]_{h(t)}$, where $\mathcal O[t]_{h(t)}$ is the localization of $\mathcal O[t]$;
\par
(ii) the restriction of $\mathcal G_t$ to $\{0\}\times U$ coincides
with the restriction of $\mathcal G$ to $U$;
%%%the original $\bG$-bundle $\mathcal G$.
\par
(iii) $\mathcal Z\cap \{1\}\times U=\emptyset$;
%%%(iii) $f(1) \in \mathcal O$ is invertible in $\mathcal O$.

%the map
%$$ H^1_{\text{\rm\'et}}(\mathcal O \otimes_k A,G)\to H^1_{\text{\rm\'et}}(K \otimes_k A,G), $$
%\noindent
%induced by the inclusion $\mathcal O$ into $K$, has trivial kernel.
\end{thm}
%%%If the field $k$ is infinite a weaker result is proved in
%%%\cite[Thm.1.2]{PSV}.

%{\bf It seems that using Fedorov's trick \cite[??]{Fed} one could extend the following result to all reductive $U$-group schemes.
%See also \cite[Section 3.3]{GPan}.
\begin{thm}{\rm{(Weak homotopy invariance)}}
\label{th:psv}
%%%Suppose $k(v)$ is an {\bf infinite} field.
Let
$U$ be as in Theorem~\ref{MainThm1}.
This time let $\bG$ be a reductive $U$-group scheme.
Let $\mathcal Z\subset\mathbb A^1_U$ be a closed subscheme finite over $U$ and such that $\mathcal Z\cap \{1\}\times U=\emptyset$.
%Let $Y\subset\mathbb A^1_U$ be a closed subscheme finite and
%\'etale over $U$ and such that\\
%(i) $\bG_Y:=\bG\times_U Y$ is isotropic, \\
%(ii) $Y\cap Z=\emptyset$, \\
%(iii) if the $u$-group scheme $\bG_u:=\bG\times_U u$ is isotropic, then there is a $k(u)$-rational point in $Y_u:=\mathbb A^1_u\cap Y$ (here $k(u)$ is the residue field of $u$). \\
Let $\mathcal G$ be a~principal $\bG$-bundle over
$\mathbb P^1_U$ such that its restriction to
$\mathbb P^1_U - \mathcal Z$ is trivial.
Then for the zero section $i_0: U\to \mathbb P^1_U$ of the projection $pr_U: \mathbb P^1_U\to U$
the $\bG$-bundle $i^*_0(\cG)$ over $U$ is trivial.
%the restriction of $\mathcal G$ to
%$\mathbb P^1_U-Y$ is also trivial.\\
%In particular, the principal $\bG$-bundle $\mathcal G$ is trivial locally for the Zarisky topology.\\
%(Note that $Y$ and $Z$ are closed in $\P^1_U$ since they are finite over $U$).
\end{thm}

%%%\begin{thm}{\rm{(Weak homotopy invariance)}}
%%%\label{th:psv}
%Let $R$ be the semi-local ring of finitely many closed points on an irreducible smooth affine variety over {\bf a finite field} $k$,
%%%Let $A$, $X$, $\mathcal O$ and $U$ be as in Theorem \ref{MainHomotopy}.
%be the local ring of a closed point on $X$.
%and let $K$ be its
%field of fractions.
%Let $\mathcal O$ be the local ring of a closed point on the $A$-scheme above.
%Set $U=\spec \mathcal O$.
%%%This time let $\bG$ be a reductive $U$-group scheme.
%(see~\cite[Exp.~XXIV, Sect.~5.3]{SGA3} for the definition).
%%%Let $\cE$ be a principal $\bG$-bundle over the affine line $\mathbb A^1_U=\spec {\mathcal O}[t]$, and let $h(t)\in \mathcal O[t]$ be a monic polynomial
%%%with $h(1)\in \cO^{\times}$.
%%%Denote by $(\mathbb A^1_U)_h$ the open subscheme in $\mathbb A^1_U$ given by $h(t)\ne0$.
%%%Suppose
%%%the restriction of $\cE$ to $(\mathbb A^1_U)_h$ is a trivial principal $\bG$-bundle.
%If a section $s:U\to\mathbb A^1_U$ of the projection $\mathbb A^1_U\to U$ is such that
%$s(U)\cap (\{1\}\times U)=\emptyset$, then the $\bG$-bundle $s^*\cE$ over $U$ is trivial.

%%%Then $i^*_0(\cE)$ is trivial, where $i_0: U\to \mathbb A^1_U$ is the zero section of $\mathbb A^1_U$.
%%%\end{thm}
{\bf Clearly, Theorem \ref{MainThm1} is a consequence of Theorems \ref{MainHomotopy} and \ref{th:psv}}.
Theorem \ref{MainHomotopy} is one of the main result of the preprint
\cite{GPan}. Theorem \ref{th:psv} is easily derived in the present preprint
from
two recent results.
{\bf Namely, from
\cite[Thm.~Main]{PSt}
and
Theorem \ref{th:psg}.
}
The latter results are
stated right below.
Theorem \ref{th:psg} is proven in this preprint.

\begin{thm}{\rm{(A finite field version of the Gille Theorem).}}
\label{th:psg}
Suppose $k(v)$ is a {\bf finite} field.
Let
$U$ be as in Theorem~\ref{MainThm1}.
This time let $\bG$ be a semi-simple $U$-group scheme.
Let $\mathcal Z\subset\mathbb A^1_U$ be a closed subscheme finite over $U$ and such that $\mathcal Z\cap \{1\}\times U=\emptyset$.
%Let $Y\subset\mathbb A^1_U$ be a closed subscheme finite and
%\'etale over $U$ and such that\\
%(i) $\bG_Y:=\bG\times_U Y$ is isotropic, \\
%(ii) $Y\cap Z=\emptyset$, \\
%(iii) if the $u$-group scheme $\bG_u:=\bG\times_U u$ is isotropic, then there is a $k(u)$-rational point in $Y_u:=\mathbb A^1_u\cap Y$ (here $k(u)$ is the residue field of $u$). \\
Let $\mathcal G$ be a~principal $\bG$-bundle over
$\mathbb P^1_U$ such that its restriction to
$\mathbb P^1_U - \mathcal Z$ is trivial.
Then there exists a closed subscheme $Y\subset \mathbb P^1_U$ finite and etale over $U$ such that
$\mathcal Z\cap Y=\emptyset$ and $\mathcal G$ is trivial over $\P^1_U-Y$.
\end{thm}

\begin{thm}{\rm\cite[Main Thm.]{PSt}}{\rm{(An infinite field version of the Gille Theorem).}}
\label{th:ps}
Suppose $k(v)$ is an {\bf infinite} field.
Let
$U$ be as in Theorem~\ref{MainThm1}.
This time let $\bG$ be a semi-simple $U$-group scheme.
Let $\mathcal Z\subset\mathbb A^1_U$ be a closed subscheme finite over $U$ and such that $\mathcal Z\cap \{1\}\times U=\emptyset$.
%Let $Y\subset\mathbb A^1_U$ be a closed subscheme finite and
%\'etale over $U$ and such that\\
%(i) $\bG_Y:=\bG\times_U Y$ is isotropic, \\
%(ii) $Y\cap Z=\emptyset$, \\
%(iii) if the $u$-group scheme $\bG_u:=\bG\times_U u$ is isotropic, then there is a $k(u)$-rational point in $Y_u:=\mathbb A^1_u\cap Y$ (here $k(u)$ is the residue field of $u$). \\
Let $\mathcal G$ be a~principal $\bG$-bundle over
$\mathbb P^1_U$ such that its restriction to
$\mathbb P^1_U - \mathcal Z$ is trivial.
Then there exists a closed subscheme $Y\subset \mathbb P^1_U$ finite and etale over $U$ such that
$\mathcal Z\cap Y=\emptyset$ and $\mathcal G$ is trivial over $\mathbb P^1_U-Y$.
\end{thm}
%By Remark \ref{red_F_yields_psv} Theorem \ref{th:psv} follows from Theorem \ref{red_to F}.
%Thus,
%{\bf to get Theorem \ref{MainThm1} it remains to prove Theorem \ref{red_to F}
%}.
%This will be done in the rest of the present preprint.
%Namely Theorems~\ref{red_to F} and \ref{MainThm1} are proved in Section \ref{sect:outline}.

%\begin{notation}
%\label{Feds_d}
%Let $\bG^{sc}_{split}$ be {\bf the split form} of the simply-connected $U$-group scheme $\tilde \bG^{ad}$.
%Let $d=[P:Q]$, where $P$ and $Q$ are the weight and the root lattices for
%the simply-connected $U$-group scheme $\bG^{sc}_{split}$.
%\end{notation}

%The following lemma is obvious. However it is very useful.
%\begin{lem}(The Fedorov trick).
%\label{Feds_trick}
%Using the notation and the hypotheses of Theorem \ref{th:psv} and Notation
%\ref{Feds_d}
%consider the $U$-morphism
%$\pi_d: \P^1_U\to \P^1_U$ given by $t\mapsto t^d$.
%Put $\cF=\pi^*_d(\cG)$.
%Then $\cF$ is a principal $\bG$-bundle on $\P^1_U$.
%Then $i^*_0(\cG)=i^*_0(\cF)$ as the principal $\bG$-bundles over $U$. \\
%Thus,
%$i^*_0(\cG)$ is a trivial over $U$ if and only if $i^*_0(\cF)$ is a trivial over $U$.
%\end{lem}

%\begin{thm}
%\label{red_to F}
%Under Notation \ref{Feds_d} and the hypotheses of Theorem \ref{th:psv} put $Z(d):=\pi^{-1}(Z)$
%and $\cF=\pi^*_d(\cG)$. Then \\
%(i) $Z(d)$ is a closed subscheme of $\mathbb A^1_U$ finite over $U$ and such that $Z(d)\cap \{1\}\times U=\emptyset$; \\
%(ii) $\mathcal F$ is a~principal $\bG$-bundle over
%$\mathbb P^1_U$ whose restriction to
%$\mathbb P^1_U- Z(d)$ is trivial; \\
%(iii) the $\bG$-bundle $i^*_0(\cF)$ is trivial.
%\end{thm}

\begin{rem}
\label{ps_and_psg}
Surely, Theorem \ref{th:psg} and \ref{th:ps} are {\bf quite similar} to each the other.
However the proof of Theorem \ref{th:ps} requires
{\bf an additional new method}.
%The assertions (i) and (ii) of Theorem \ref{red_to F} are obvious.
%The assertion (iii) is the Key one.
%The assertion (iii) together with Lemma \ref{Feds_trick} yield Theorem \ref{th:psv}.
\end{rem}

To state Proposition \ref{SchemeY22} recall the following notion.
Let $\bH$ be a semi-simple group scheme over a semi-local scheme $W$.
%and such that the adjoint
%$W$-group scheme $\bH^{ad}$ is simple. Recall that $\bH$ is called isotropic if
Recall that $\bH$ is called quasi-split, if
%%%the restriction of $\bH$ to each connected component of $W$ contains a proper parabolic
%%%subgroup scheme.
%Equivalently, the reductive group scheme $\bH$ is called isotropic if
%the restriction of $\bH$ to each connected component of $W$ contains a proper parabolic
%subgroup scheme.
%%%Note that by \cite[Exp. XXVI, Cor. 6.14]{D-G} this is equivalent
%%%to the usual definition, that is to the requirement that
the $W$-group scheme $\bH$ contains a Borel group subscheme.

\begin{prop}
\label{SchemeY22}
Suppose the residue field $k(v)$ is finite. Let $U$ be as in Theorem~\ref{th:psv} and
$\mathcal Z\subset\mathbb A^1_U$ be an arbitrary closed subset finite over $U$.
This time suppose $\bG$ is a semi-simple $U$-group scheme.
%Let $\mathcal G$ be a~principal $\bG$-bundle over
%$\mathbb P^1_U$ such that its restriction to
%$\mathbb P^1_U - \mathcal Z$ is trivial.
%%Suppose additionally that the adjoint $U$-group scheme $\bG^{ad}$ is simple ({\bf What is it for ???}).
%Let $\mathcal Z$ be a closed subscheme finite over $U$ and such that $\mathcal Z\cap \{1\}\times U=\emptyset$.
Then there is a closed subscheme $Y\subset\mathbb A^1_U$ which is \'etale and finite over $U$ and such that \\
%\begin{itemize}
(i) the $Y$-group scheme $\bG_Y:=\bG\times_U Y$ is quasi-split;\\
%%%(ii) the $\bG_Y$-bundle $\cG|_Y$ is trivial.\\
(ii) $Y\cap \mathcal Z=\emptyset$, $Y\cap (\{0\}\times U)=\emptyset$, $Y\cap (\{1\}\times U)=\emptyset$; \\
(iii) for any closed point $u \in U$ one has $Pic(\P^1_u - Y_u)=0$, where $Y_u:=\P^1_u\cap Y$;\\
%\end{itemize}
(Note that $Y$ and $\mathcal Z$ are closed in $\P^1_U$ since they are finite over $U$).
\end{prop}

The preprint is organized as follows.
%%%By Remark \ref{red_F_yields_psv} Theorem \ref{th:psv} follows from Theorem \ref{red_to F}.
In Section \ref{useful} a useful lemma is proved, namely Lemma \ref{F1F2}.
In Section \ref{SchemeY_section} Propositions \ref{SchemeY22} is proved.
In Section \ref{semi-simple} Theorem \ref{th:psg} is proved.
Theorem \ref{th:psv} is derived from
\cite[Thm. Main]{PSt} and Theorem \ref{th:psg} in Section
\ref{sect:outline}.
Theorem \ref{MainThm1}
is proved at the very end of Section
\ref{sect:outline}.
%Theorem~\ref{th:psv} is proved in Section \ref{sec: proof_of_psv}.
%Theorem \ref{MainThm1}
%is proved in Section \ref{Reducing_MainThm1_to_two_others}
%as a consequence of Theorems~\ref{MainHomotopy} and \ref{th:psv}.
\subsection*{Acknowledgements} 
The authors thank the excellent environment of the International Mathematical Center at POMI.
This work is supported by Ministry of Science and Higher Education of the Russian Federation, agreement \textnumero~ 075-15-2022-289.

\section{A useful lemma}\label{useful}
We will need the following lemma to prove Proposition \ref{SchemeY22}. This lemma itself
is proved at the end of the present section.
%Let $k$ be a {\bf finite} field. Let $\mathcal O$ be the semi-local ring of finitely many {\bf closed points} on a
%$k$-smooth irreducible affine $k$-variety $X$.
%Set $U=\spec \mathcal O$. Let ${\bf u}\subset U$ be the set of all closed points in $U$.
%For a point $u\in {\bf u}$ let $k(u)$ be its residue field.
%It is a finite extension of the field $k$. Note that $\cO=\Gamma(U,\cO_U)$.\\\\
Suppose $k(v)$ is a {\bf finite} field. Write $k$ for $k(v)$ in this Section.
Let $\mathcal O$ be the semi-local ring of finitely many {\bf closed point} on a
$V$-smooth irreducible $V$-projective $V$-scheme $X$.
Set $U=\spec \mathcal O$.
Let ${\bf u}\subset U$ be the set of all closed points in $U$.
For a point $u\in {\bf u}$ let $k(u)$ be its residue field.
It is a finite extension of the field $k$
and thus it is finite. Note that $\cO=\Gamma(U,\cO_U)$.
%%%Let ${\bf u}\subset U$ be the set of all closed points in $U$.
%For the closed point $u\in U$ let $k(u)$ be its residue field.
%It is a finite extension of the field $k$ and thus it is finite.
%Note that $\cO=\Gamma(U,\cO_U)$.
The following lemma is a version in our context of
\cite[Lemma 3.1]{Pan3}. It is proved at the end of this Section.

\begin{lem}
\label{F1F2}
%Let $U$ be as in the Proposition \ref{SchemeY}.
Let $\mathcal Z\subset\mathbb A^1_U$ be a closed subscheme finite over $U$.
Let $Y^{\prime} \to U$ be a finite \'{e}tale morphism such that
for any closed point $u_i$ in $U$ the fibre $Y^{\prime}_{u_i}$ of $Y^{\prime}$ over $u_i$
contains a $k(u_i)$-rational point. Then there are finite field extensions
$k_1$ and $k_2$ of the finite field $k$
and finite etale extensions $A_1/A$ and $A_2/A$ with local $A$-algebras $A_1,A_2$ with the residue
fields $k_1$ and $k_2$ respectively and
such that \\
(i) the degrees $[k_1: k]$ and $[k_2: k]$ are coprime,\\
(ii) $k(u_i) \otimes_k k_r$ is a field for $r=1$ and $r=2$,\\
(iii) the degrees $[k_1: k]$ and $[k_2: k]$ are strictly greater than any of the degrees
$[k(z): k(u)]$, where $z$ runs over all closed points of $\mathcal Z$,\\
(iv) there is a closed embedding of $U$-schemes
$Y^{\prime\prime}=((Y^{\prime}\otimes_A A_1) \coprod (Y^{\prime}\otimes_A A_2)) \xrightarrow{i} \mathbb A^1_U$,\\
(v) for $Y=i(Y^{\prime\prime})$ one has $Y \cap \mathcal Z = \emptyset$, $Y\cap (\{0\}\times U)=\emptyset$, $Y\cap (\{1\}\times U)=\emptyset$; \\
(vi) for any closed point $u_i$ in $U$ one has
$Pic(\P^1_{u_i}-Y_{u_i})=0$.
\end{lem}

\begin{notation}\label{notn: F1F2}
Let $k$ be the finite field of characteristic $p$,
$k'/k$ be a finite field extensions.
Let $c=\sharp(k)$ (the cardinality of $k$).
For a positive integer $r$ let $k'(r)$
%(respectively $\mathbb L(r)$, respectively $\mathbb E(r)$)
be a unique field extension of the degree $r$ of the field $k'$.
%(respectively of $\mathbb L$, respectively of $\mathbb E$).
%%%Let $\mathbb A^1_{k}(r)$ be the set of all degree $r$ points
%%%on the affine line $\mathbb A^1_{k}$.
%%%Let $Irr(r)$ be the number of the degree $r$ points on
%%%$\mathbb A^1_{k}$.
\end{notation}

\begin{notation}
\label{d(Y_u)}
For any \'{e}tale $k$-scheme $W$
set $d(W)=\text{max} \{ deq_{k} k(v)| v\in W \}$.
\end{notation}

\begin{lem}
\label{F1F2_copy_very_final}
Let $k$ be the finite field from the lemma \ref{F1F2}.
Let $U$ be the the semi-local scheme as in the lemma \ref{F1F2}.
Let $u \in U$ be a closed point and let $k(u)$ be its residue field (it is a finite extension of the finite field $k$).
Let $\mathcal Z_u\subset\mathbb A^1_u$ be a closed subscheme finite over $u$.
Let $Y^{\prime}_u \to u$ be a finite \'{e}tale morphism such that
$Y^{\prime}_u$
contains a $k(u)$-rational point. Then for any
two different primes $q_1,q_2\gg 0$
the finite field extensions
$k(q_1)$ and $k(q_2)$ of the finite field $k$ satisfy the following conditions \\
%(i) the degrees $[k_1: k]$ and $[k_2: k]$ are coprime,\\
(i) for $j=1,2$ one has $q_j > d(Y'_u)$;\\
(ii) for any $j=1,2$ and any point $v\in Y'_u$ the $k$-algebra $k(v) \otimes_k k(q_j)$ is a field;\\
(iii) both primes $q_1,q_2$ are strictly greater than
$\text{max}\{[k(z): k]| z\in \mathcal Z_u \}$;\\
%than {\bf any of the degrees}
%$[k(z): k]$ , where $z$ runs over all points of $Z_u$,\\
(iv) there is a closed embedding of $u$-schemes
$Y^{\prime\prime}_u=((Y^{\prime}_u\otimes_k k_1) \coprod (Y^{\prime}_u\otimes_k k_2)) \xrightarrow{i_u} \mathbb A^1_u$,\\
(v) for $Y_u=i_u(Y^{\prime\prime}_u)$ one has $Y_u \cap \mathcal Z_u = \emptyset$, $Y_u\cap (\{0\}\times u)=\emptyset$, $Y_u\cap (\{1\}\times u)=\emptyset$;\\
(vi) $Pic(\P^1_{u}-Y_u)=0$.\\
%%%there is a closed embedding of $U$-schemes
%%%$Y^{\prime\prime}=((Y^{\prime}\otimes_k k_1) \coprod (Y^{\prime}\otimes_k k_2)) \xrightarrow{i} \mathbb A^1_U$,\\
\end{lem}

\begin{proof}[Proof of the lemma \ref{F1F2}]
We prove the lemma for the case of local $U$ and leave the general case to the reader.
So, we may suppose that there is only one closed point $u$ in $U$.
Let $k(u)$ be its residue field.
Let $Y^{\prime} \to U$ be the finite \'{e}tale morphism from the lemma \ref{F1F2}.
%%%By the hypotheses of the lemma \ref{F1F2}
%%%the fibre $Y^{\prime}_{u}=Y'\times_U u$ of $Y^{\prime}$ over $u$
%%%contains a $k(u)$-rational point.
Let $Y^{\prime}_{u}=Y'\times_U u$ be the fibre of $Y^{\prime}$ over $u$.
By the hypotheses of the lemma \ref{F1F2}
the $k(u)$-scheme $Y^{\prime}_{u}$ contains a $k(u)$-rational point.
Let $k_1/k$ and $k_2/k$ be the field extensions from
the lemma \ref{F1F2_copy_very_final}.
Clearly, there are finite etale extensions $A_1/A$ and $A_2/A$
with local $A$-algebras $A_1$ and $A_2$
with the mentioned residue
fields $k_1$ and $k_2$.
%%%as in lemma \ref{F1F2_copy_very_final}.
Set
$Y''=(Y^{\prime}\otimes_A A_1) \sqcup (Y^{\prime}\otimes_A A_2)$
and
$Y''_u=Y''\times_U u$.
Then one has
$Y''_u=(Y^{\prime}_u\otimes_k k_1) \sqcup (Y^{\prime}_u\otimes_k k_2)$.\\\\
Let
$i_u: Y''_u \hookrightarrow \mathbb A^1_u$
be the closed embedding from the lemma
\ref{F1F2_copy_very_final}. Then the pull-back map of the $k(u)$-algebras
$i^*_u: k(u)[t]\to \Gamma(Y''_u,\cO_{Y''_u})$ is surjective.
Consider the element $b=i^*_u(t)\in \Gamma(Y''_u,\cO_{Y''_u})$.
Let $B\in \Gamma(Y'',\cO_{Y''})$ be a lift of the element $b$.

Consider a unique $\cO$-algebra homomorphism
$\cO[t]\xrightarrow{I^*} \Gamma(Y'',\cO_{Y''})$
which takes $t$ to $B$.
%Denote it $i^*$.
Since $\Gamma(Y'',\cO_{Y''})$ is a finitely generated $\cO$-module
and $I^*\otimes_{\cO} k(u)=i^*$,
the Nakayama lemma shows that the map $I^*$ is surjective.
Hence the induced scheme morphism
$I: Y''\to \mathbb A^1_U$
is a closed embedding.
The assertion (iv) is proved.
Set $Y=I(Y'')$.

Clearly, the two closed embeddings
$Y''_u\to Y''\xrightarrow{I} \mathbb A^1_U$
and
$Y''_u\xrightarrow{i_u} \mathbb A^1_u \to \mathbb A^1_U$
coincide.
Lemma \ref{F1F2_copy_very_final} completes now the proof of the lemma
\ref{F1F2}.
\end{proof}

\section{Proof of Proposition \ref{SchemeY22}}
\label{SchemeY_section}
%Now we assume that Theorem~\ref{MainThm2} is true.
%The main issue of this section is to prove Proposition
%\ref{SchemeY}.
%%%Theorem \ref{th:psv} in the case of an infinite field $k$ is proved in \cite[Thm. 2]{FP}.
%%%{\it So, we prove the remaining case, when the field $k$ is finite.
%%%}
%IT IS EVEN A PARTIAL CASE OF Lemma \cite[Lemma 8.3]{P1}.
The main aim of this section is to prove Proposition \ref{SchemeY22}.
%%%that
%%%asserts existence of a closed subscheme $Y$ in $\mathbb A^1_U$ subjecting
%%%the hypotheses of Theorem \ref{MainThm22}.
Recall the following result \cite[Lemma 3.1]{Pan3} which is a simplified version of
\cite[Lemma 3.3]{Pan1}.
\begin{lem}(\cite[Lemma 3.1]{Pan3})
\label{OjPan}
Let $S=\text{Spec}(R)$ be a regular
semi-local
{\bf irreducible
}
scheme
such that
the residue field at any of its closed points is finite.
Let $T$ be a closed
subscheme of $S$. Let $W$ be a closed subscheme of the projective space
$\Bbb P^n_S$.
%and
%$X=\bar X\cap \Bbb A^d_S$, where $\Bbb A^d_S$ is the affine space defined by
%$X_0\neq0$. Let
%$X_{\infty}=\bar X\setminus X$ be the intersection of $\bar X$ with the
%hyperplane at infinity
%$X_0=0$.
Assume that over $T$ there exists a section
$\delta:T\to W$
of the canonical projection $W\to S$. Assume further that
%\smallskip
%{\rm{(1)}}
$W$ is smooth and equidimensional over $S$,
of relative dimension $r$.
%{\rm{(2)}} For every closed
%point $s\in S$ the closed fibres of $X_\infty$ and $X$
%satisfy
%$$\text{dim}(X_{\infty}(s))<\text{dim}(X(s))=r\;.$$
%\smallskip
Then there exists a closed subscheme $\tilde S$ of $W$ which is finite
\'etale over
$S$ and contains $\delta(T)$.
\end{lem}

\begin{proof}[Proof of Lemma \ref{OjPan}]
Formally this lemma is a bit different of the one \cite[Lemma 3.3]{Pan1}.
However it can be derived from \cite[Lemma 3.3]{Pan1} as follows.
There is a Veronese embedding $\Bbb P^n_S\hookrightarrow \Bbb P^d_S$
and linear homogeneous coordinates $X_0,X_1,...,X_d$ on $\Bbb P^d_S$
such that the following holds:\\
\smallskip
{\rm{(1)}}
if
$W_{\infty}=W\cap \{X_0=0\}$ is the intersection of $W$ with the
hyperplane
$X_0=0$, then for every closed
point $s\in S$ the closed fibres of $W_\infty$
satisfy
$\text{dim}(W_{\infty}(s))<r=\text{dim}(W(s))$; \\
{\rm{(2)}}
$\delta$ maps $T$ into
the closed subscheme of $\Bbb P^d_T$ defined by
$X_1=\dots=X_d=0$.\\
\smallskip
Applying now \cite[Lemma 3.3]{Pan1} to $\bar X=W$, $X_{\infty}=W_{\infty}$ and $X=W-W_{\infty}$
we get a closed subscheme $\tilde S$ of $W-W_{\infty}$ which is finite
\'etale over
$S$ and contains $\delta(T)$. Since $\tilde S$ is finite over $S$ it is closed in $W$ as well.
\end{proof}
%Let $U$ and $\bG$ be as in Proposition~\ref{SchemeY22}.

%Consider the reduced closed subscheme $\bf u$ of $U$,
%%%equals the point $u$.
%Thus
%\begin{equation}\label{bf_u}
% \bu\cong\coprod_i\spec k(u_i).
%\end{equation}
%Set $\bG_{\bf u}=\bG\times_U {\bf u}$.
%%%By $\bG_{u_i}$ we denote the fiber of $\bG$ over $u_i$;
%It is a simple simply-connected algebraic group over $\bf u$.

%\begin{prop}
%\label{SchemeY}
%Let $Z\subset\mathbb A^1_U$ be a closed subscheme finite over $U$.
%There is a closed subscheme $Y\subset\mathbb A^1_U$ which is \'etale and finite over $U$ and such that \\
%%\begin{itemize}
%(i) $\bG_Y:=\bG\times_UY$ is quasi-split, \\
%(ii) $Y\cap Z=\emptyset$,\\
%(iii) for any closed point $u \in U$ one has $Pic(\P^1_u - Y_u)=0$, where $Y_u:=\P^1_u\cap Y$.\\
%%%(Note that $Y$ and $Z$ are closed in $\P^1_U$ since they are finite over $U$).
%There is a closed subscheme $Y\subset\P^1_U$ such that $Y$ is \'etale over $U$,
%$\bG_Y=\bG\times_UY$ is isotropic, and for all $u_i\in\bu'$ there is a $k(u_i)$-rational point $y_i\in Y$ lying over $u_i$.
%\end{prop}

\begin{proof}(of Proposition \ref{SchemeY22}).
Let $U$ and $\bG$ be as in Proposition~\ref{SchemeY22} and
let $u_1,\ldots,u_n$ be all the closed points of $U$. Let $k(u_i)$ be the residue field of $u_i$.
Since $k(v)$ is finite the fields $k(u_i)$ are all finite.
Consider the reduced closed subscheme $\bu$ of $U$,
whose points are $u_1$, \ldots, $u_n$. Thus
\begin{equation}\label{bf_u}
 \bu\cong\coprod_i\spec k(u_i).
\end{equation}
Set $\bG_\bu=\bG\times_U\bu$. By $\bG_{u_i}$ we denote the fiber of $\bG$ over $u_i$.
%%%it is a simple simply-connected algebraic group over $k(u_i)$.
%%%\begin{proof}
%If $\bu'$ is empty, we just take $Y=\emptyset$.
For every $u_i$ in $\bu$ choose a Borel subgroup $\bB_{u_i}$ in $\bG_{u_i}$.
{\it The latter is possible since the fields $k(u_i)$ are finite.}
Let $\cB$ be the $U$-scheme of Borel subgroup schemes of $\bG$.
%of the same type as $\bP_{u_i}$.
It is a smooth projective $U$-scheme (see~\cite[Cor.~3.5, Exp.~XXVI]{SGA3}).
The subgroup $\bB_{u_i}$ in $\bG_{u_i}$ is a $k(u_i)$-rational point~$b_i$ in the fibre of $\cB$ over the point $u_i$.
Now apply Lemma
\ref{OjPan} to
the scheme
$U$ for $S$,
the scheme $\mathbf u$ for $T$,
the scheme $\cB$ for $W$
and to the section $\delta: \mathbf u \to \cB$,
which takes the point $u_i$ to the point $b_i \in \cB$.
Since the scheme $\cB$ is $U$-smooth and
equi-dimensional over $U$
we are under the assumption of Lemma \ref{OjPan}. Hence
there is a closed subscheme $Y^{\prime}$ of $\cB$ such that
$Y^{\prime}$ is finite \'etale over $U$ and all the $b_i$'s are in $Y$.
The $U$-scheme $Y^{\prime}$ satisfies the hypotheses of Lemma
\ref{F1F2}.

Take the closed subscheme $Y$ of $\mathbb A^1_U$ as in the item (v)
of Lemma \ref{F1F2}. For that specific $Y$ the conditions (ii) and (iii) of the Proposition are
obviously satisfied. The condition (i) is satisfied too, since already
it is satisfied for the $U$-scheme $Y^{\prime}$. This proves the Proposition.
\end{proof}

\section{Proof of Theorem \ref{th:psg}}
\label{semi-simple}
%In this Section $\cF$ is the principal $\bG$-bundle as in \ref{Feds_trick}. We will construct in this
%Section a principal $\bG$-bundle $\cF_{mod}$ over $\P^1_U$ with certain properties as in
%Propositions \ref{F_mod_for_ss_case} and \ref{F_mod_for_red_case}.
%Theorems~\ref{red_to F} and \ref{MainThm1} will be proved in Section \ref{sect:outline}
%using
%Propositions \ref{F_mod_for_ss_case} and \ref{F_mod_for_red_case}.
%%%Suppose additionally in this Section that the $U$-group scheme $\bG$ as in Theorem \ref{th:psg} is semi-simple.
The aim of this Section is to prove Theorem \ref{th:psg}.

For simplicity consider only the case of local $U$. Let $u\in U$ be its unique closed point.
%Let $\bG_u:=\bG\times_U u$ be the fibre of $\bG$ over $u$.
Let $\Pi: \bG^{sc}\to \bG$ be the simply-connected cover of the semi-simple $\bG$
and
$\Pi_u: \bG^{sc}_u\to \bG_u$ be the group morphism on the closed fibres.
It is well-known that $\Pi_u$ is the simply-connected cover of the semi-simple $\bG_u$.
Put $\mu=Ker(\Pi)$. Then $\mu$ is in center of $\bG^{sc}$.
%By construction of $\cF$ and due to the Gille theorem \cite{GilleTorseurs} there exists
%a Nisnevich locally trivial $\bG^{sc}_u$-bundle $\cF^{sc}_u$ over $\P^1_u$ such that
%the $\bG_u$-bundle $\Pi_{u,*}(\cF^{sc}_u)$ is isomorphic to the $\bG_u$-bundle $\F_u$.

Note that the obstruction to a lift of $\cG$ up to a $\bG^{sc}$-bundle over $\P^1_U$ vanishes. Thus,
there is a $\bG^{sc}$-bundle $\cG^{sc}$ over $\P^1_U$ such that $\cG=\Pi_*(\cG^{sc})$.
This choice of $\cG^{sc}$ can be "sheafted" by any $\mu$-bundle over $\P^1_U$.
Using this observation we may and will suppose that additionally $\cG^{sc}$ is such that
$\cG^{sc}|_{1\times U}$ is trivial.

In this case the $\cG^{sc}_u$ is trivial at the point $1$ of $\P^1_u$.
Thus, by the Gille theorem \cite{GilleTorseurs}
the restriction $\cG^{sc}_u$ to $\P^1_u$ is a Nisnevich locally trivial $\bG^{sc}_u$-bundle.

Let $Y$ be the closed subscheme in $\mathbb A^1_U$ as in Proposition \ref{SchemeY22}.
We claim that the $\bG^{sc}$-bundle $\cG^{sc}|_Y$ is trivial. To check this recall that
$Y\subset \P^1_U - \mathcal Z$ and thus, $\cG^{sc}|_Y$ is of the form
$in_*(\mathcal K)$ for a $\mu_Y$-bundle $\mathcal K$ over $Y$.
Since $\mu_Y$ is contained in each maximal $Y$-torus $T$ of $\bG^{sc}_Y$
the class $[\cG^{sc}|_Y]\in H^1_{et}(Y,\bG^{sc})=H^1_{et}(Y,\bG^{sc}_Y)$
comes from $H^1_{et}(Y,T)$.
Since $\bG^{sc}_Y$ is quasi-split, we may choose a quasi-split maximal $Y$-torus $T$ of $\bG^{sc}_Y$.
In this case $H^1_{et}(Y,T)=*$ and, thus $\cG^{sc}|_Y=in_*(\mathcal K)$ is a trivial $\bG^{sc}$-bundle.
As a consequence the $\bG^{sc}$-bundle $\cG^{sc}|_{Y^h}$ is trivial too.

Put $(\cG^{sc})'=\cG^{sc}|_{\P^1_U - Y}$. Since the $\bG^{sc}$-bundle $\cG^{sc}|_{Y^h}$ is trivial
we can present $\cG^{sc}$ in the form $\cE((\cG^{sc})', \phi)$,
where $\phi: \dot Y^h\times_U \bG^{sc}\to \cG^{sc}|_{\dot Y^h}$
is a $\bG^{sc}_Y$-bundle isomorphism. In this case the $\cG^{sc}_u$ is presented in the form $\cE((\cG^{sc})'_u, \phi_u)$,
where $\phi_u: \dot Y^h_u\times_u \bG^{sc}_{u}\to \cG^{sc}|_{\dot Y^h_u}$
is a $\bG^{sc}_u$-bundle isomorphism.

Recall that $\cG^{sc}_u$ is a Nisnevich locally trivial $\bG^{sc}_u$-bundle over to $\P^1_u$.
Recall that $Pic(\P^1_u-Y_u)=0$.  Thus, by the Gille Theorem \cite{GilleTorseurs} $(\cG^{sc}_u)|_{\P^1_u-Y_u}$
is trivial. This yields that there is an element $\alpha_u\in \bG^{sc}_u(\dot Y^h_u)$
such that the $\bG^{sc}_u$-bundle $\E((\cG^{sc})'_u, \phi_u\circ \alpha_u)$ over $\P^1_u$
{\bf is trivial}.

The $u$-group scheme $\bG^{sc}_u$ is simply connected. Thus, the group
$\bG^{sc}_u(\dot Y^h_u)$ equals to the one $\bE(\bG^{sc}_u(\dot Y^h_u))$.
The scheme $\dot Y^h$ is affine. Thus, the group homomorphism
$$\bE(\bG^{sc}(\dot Y^h))\to \bG^{sc}_u(\dot Y^h_u)$$
is surjective. As a consequence the group homomorphism
$\bG^{sc}(\dot Y^h)\to \bG^{sc}_u(\dot Y^h_u)$
is surjective. Thus, there is an element
$\alpha\in \bG^{sc}(\dot Y^h)$
such that $\alpha|_{\dot Y^h_u}=\alpha_u$ in $\bG^{sc}_u(\dot Y^h_u)$.
Put
$$\cG^{sc}_{mod}=\cE((\cG^{sc})', \phi\circ \alpha) \ \ \text{and} \ \ \cG^{mod}=\Pi_*(\cG^{sc}_{mod}).$$
We already know that the $\bG^{sc}_u$-bundle
$\E((\cG^{sc})'_u, \phi_u\circ \alpha_u)$ over $\P^1_u$
is trivial over $\P^1_u$. Thus,
the $\bG_u$-bundle $\cG^{mod}_u$ is trivial too.
By the main result of \cite{Tsy} the $\bG$-bundle $\cG_{mod}$
is extended from $U$ via the pull-back along the projection
$pr_U: \P^1_U\to U$. On the other hand $\cG^{mod}|_{1\times U}$
is trivial, since $1\times U\subset \P^1_U-Y$, $\cG^{mod}|_{\P^1_U-Y}=\cG|_{\P^1_U-Y}$ and $\cG|_{\P^1_U-Y}$ is trivial.
Hence the $\bG$-bundle $\cG^{mod}$ is trivial over $\P^1_U$.
Since $\cG|_{\P^1_U - Y}=\cG^{mod}_{\P^1_U - Y}$
it follows that
the $\bG$-bundle $\cG|_{\P^1_U - Y}$ is trivial.
%%%Since $\cG=\Pi_*(\cG^{sc})$ it follows that the $\bG$-bundle $\cG|_{\P^1_U - Y}$ is trivial.
{\bf The theorem \ref{th:psg} is proved.
}

\section{Proof of Theorem \ref{th:psv} in the case of finite field $k(v)$}\label{sect:outline}
%We will consider only the case when the group $\bG^{ad}$ is simple and left the general case to the reader.
{\bf In this section we prove Theorem~\ref{th:psv} in the case of finite field $k(v)$.
We also prove at the very end of the present Section Theorem \ref{MainThm1}.
}

Begin with a proof of Theorem \ref{th:psv} in the case of finite field $k(v)$.
If $\bG$ is semi-simple and the field $k(v)$ is finite,
then Theorem~\ref{th:psv} is an easy consequence of Theorem \ref{th:psg} proven in Section \ref{semi-simple}.
Thus, we may and will
suppose that $\bG$ is not semi-simple. In this case the radical $Rad(\bG)$ is a positive rank $U$-torus.

Write $\bT$ for $Rad(\bG)$ below in this Section. Then $\bT$ is contained in the center of $\bG$ and the $U$-group
scheme $\bG^{ss}:=\bG/\bT$ is semi-simple. Consider the $U$-group morphism $q: \bG \to \bG^{ss}$.
Put $\cG^{ss}=q_*(\cG)$ (that is $\cG^{ss}$ is a $\cG^{ss}$-bundle obtained from $\cG$ by the group change
via the $U$-group morphism $q: \bG \to \bG^{ss}$). Let $\Pi: \bG^{sc}\to \bG^{ss}$ be the simply connected cover
of $\bG^{ss}$ and $j: \bG^{ss}\to \bG^{der}\subset \bG$ be the canonical $U$-group scheme morphism.
Then $q\circ j=\Pi$. Let $in: \bT\hookrightarrow \bG$ be the inclusion.

Put $\mathbf S:=Corad(\bG)=\bG/\bG^{der}$ and let $can: \bG \to \mathbf S$ be the canonical morphism.
It is known that the composite morphism
$can \circ in: \bT=Rad(\bG)\to Corad(\bG)=\mathbf S$
is an isogeny. Using this fact one can find
$d: \P^1_U\to \P^1_U, \ \ t\mapsto t^d, \ d\geq 1$
and a $\bT$-bundle $\cT_0$ semi-locally trivial for the Zariski topology
%%%a $\bT$-bundle $\cT_0$
such that
%%%$\cT_0|_{$ is trivial and
$can_*([\cT_0]\cdot d^*([\cG]))=0$ in $H^1_{et}(\P^1_U,\mathbf S)$.

Clearly, $d^*(\cG)$ is trivial over $\P^1_U-d^{-1}(\mathcal Z)$.
%%%Clearly, .
Moreover, there is $\mathcal Z'$ in $\mathbb A^1_U$ finite over $U$, containing
$d^{-1}(\mathcal Z)$ such that $\cT_0|_{\P^1_U-\mathcal Z'}$ is trivial and
$\mathcal Z'\cap \{1\}\times U=\emptyset$.
Replacing $\cG$ with $[\cT_0]\cdot d^*([\cG])$
%%%and $\mathcal Z$ with $\mathcal Z'$
and using equalities
$i^*_0([\cT_0])=0$ and $d\circ i_0=i_0$
we see that in the case of finite field $k(v)$
the following result yields Theorem~\ref{th:psv} .
\begin{prop}\label{red_case}
Under the hypotheses of Theorem \ref{th:psv} suppose additionally the field $k(v)$ is finite.
let $\cG$ be a $\bG$-bundle over $\P^1_U$ such that $can_*(\cG)$ is a trivial $\mathbf S$-bundle
and $\cG|_{\P^1_U-\mathcal Z'}$ is trivial, then $i^*_0(\cG)$ is trivial.
\end{prop}

\begin{proof}[Proof of Proposition \ref{red_case}]
Since $can_*(\cG)$ is a trivial $\mathbf S$-bundle there is an element
$\gamma \in H^1_{et}(\P^1_U,\bG^{der})$ such that $in_*(\gamma)=[\cG]$ in $H^1_{et}(\P^1_U,\bG)$.
The $U$-group scheme $\bG^{der}$ is semi-simple.
Let $Y$ be the closed subscheme in $\mathbb A^1_U$ as in Proposition \ref{SchemeY22}.

Point out here, that the finiteness of $k(v)$ {\bf guaranties} that the $k(v)$-group
$\bG_{v}$ is quasi-split.
{\bf It is quite plausible that Proposition \ref{SchemeY22}
is not true if $\bG_{v}$ is not quasi-split.}

Let $\bG^{der}(tw)$ be the group $\bG^{der}$ twisted via the element $\gamma_1:=pr^*_U(i^*_1(\gamma))$.
Let $\bG(tw)$ be the group $\bG$ twisted via the element $in_*(\gamma_1)$.
There are set bijections
$$H^1_{et}(\P^1_U,\bG^{der}) \xrightarrow{Twist^{der}_{\P^1\times U}} H^1_{et}(\P^1_U,\bG^{der}(tw)),$$
$$H^1_{et}(U,\bG^{der}) \xrightarrow{Twist^{der}_U} H^1_{et}(U,\bG^{der}(tw)).$$
Put
$\gamma_{tw}=Twist_{\P^1\times U}(\gamma)$.
Since $Twist^{der}_U\circ i^*_1=i^*_1\circ Twist^{der}_{\P^1\times U}$ and
$Twist^{der}_U(i^*_1(\gamma))=*$ we conclude that
$$i^*_1(\gamma_{tw})=*\in H^1_{et}(U,\bG^{der}(tw)).$$
Let $\cG^{der}(tw)$ be a $\bG^{der}(tw)$-bundle over $\P^1_U$ with $[\cG^{der}(tw)]=\gamma_{tw}$.
The equality $*=i^*_1(\gamma_{tw})$ shows that the $\bG^{der}(tw)$-bundle $i^*_1(\cG^{der}(tw))$ is trivial.
Now by the Gille Theorem \cite{GilleTorseurs}
the $\bG^{der}(tw))_u$-bundle $\cG^{der}(tw)_u$ is Nisnevich locally trivial.
Thus, there is a $\bE\bG^{der}(tw)$- modification $\cG^{der}(tw)_{mod}$ along $Y$ of the
$\bG^{der}(tw)$-bundle $\cG^{der}(tw)$ such that $\cG^{der}(tw)_{mod}$ is trivial.

In this case $\Pi_*(\cG^{der}(tw)_{mod})$ is trivial and it is
a $\bE\bG^{ss}(tw)$- modification $\cG^{ss}(tw)_{mod}$ along $Y$ of the
$\bG^{ss}(tw)$-bundle $\cG^{ss}(tw)$.

If we put $\cG(tw)_{mod}=in_*(\cG^{der}(tw)_{mod})$, then $\cG(tw)_{mod}$ is
a $\bE\bG(tw)$- modification along $Y$ of the
$\bG(tw)$-bundle $\cG(tw)$.
Moreover, \\
(a) $q_*(\cG(tw)_{mod})=\cG^{ss}(tw)_{mod}$ is trivial;\\
(b) $\cG(tw)_{mod}|_{\P^1_U-Y}=\cG(tw)|_{\P^1_U-Y}$. \\
By the property (a) there is a $\bT$-bundle $\T$ over $\P^1_U$ such that
$\cG(tw)_{mod}=in_*(\cT)$. Since $\bT$ is a $U$-torus the $\bT$-bundle
$\cT|_{\mathbb A^1_U}$ is extended from $U$. Thus, the $\bG(tw)$-bundle
$\cG(tw)_{mod}|_{\mathbb A^1_U}$
is extended from $U$.
Recall that $\{1\}\times U$ is in $\P^1_U-Y$. Now the property (b) shows that
$\cG(tw)_{mod}|_{\mathbb A^1_U}$ is a trivial $\bG(tw)$-bundle.
Particularly, $\cG(tw)_{mod}|_{\mathbb A^1_U-Y}$ is trivial.
Since $\cG(tw)_{mod}$ is a modification of $\cG(tw)$ along $Y$ it follows that
$$\cG(tw)_{mod}|_{\mathbb A^1_U-Y}=\cG(tw)|_{\mathbb A^1_U-Y}.$$
Thus, $\cG(tw)|_{\mathbb A^1_U-Y}$ is a trivial $\bG(tw)$-bundle.
Since $\{0\}\times U$ is in $\P^1_U-Y$ it follows that
$i^*_0(\cG(tw))$ is a trivial $\cG(tw)$-bundle.
{\bf Applying now the set isomorphism
$Twist^{-1}_U: H^1_{et}(U,\bG(tw))\to H^1_{et}(U,\bG)$
we get an equality
$$i^*_0(in_*(\gamma))=i^*_1(in_*(\gamma)) \in H^1_{et}(U,\bG).$$
Since $in_*(\gamma)=[\cG]\in H^1_{et}(U,\bG)$ and $\{1\}\times U$ is in $\P^1_U-\mathcal Z'$ we get
$$i^*_1(in_*(\gamma))=[\cG|_{\{1\}\times U}]=* \in H^1_{et}(U,\bG).$$
Thus, $i^*_0([\cG])=* \in H^1_{et}(U,G)$. Proposition \ref{red_case} is proved.
Theorem \ref{th:psv} in the case of finite residue field $k(v)$ is proved.
}
\end{proof}

\begin{proof}[Proof of Theorem~\ref{th:psv}]
If the $U$-group scheme $\bG$ is semi-simple the result follows easily from
\cite[Thm. Main]{PSt} (see Theorem \ref{th:ps}) and Theorem \ref{th:psg}.

Suppose the radical of the $U$-group scheme $\bG$ is non-trivial.
The finite field $k(v)$ case of Theorem~\ref{th:psv} is already derived from Theorem \ref{th:psg}
in this Section.

The infinite field $k(v)$ case of Theorem~\ref{th:psv}
is proved in in \cite{PSt}.
Namely, it is derived there
from \cite[Thm. Main]{PSt} (= Theorem \ref{th:ps}) using arguments which are {\bf quite similar} to
the above arguments in the present Section.
{\bf However in \cite{PSt} it is not possible to use Proposition \ref{SchemeY22},
since most probably this Proposition is not true if the field $k(v)$ is infinite.
This is why Proposition \ref{SchemeY22} is replaced in \cite{PSt} with another one.
}

Theorem~\ref{th:psv} is proved.\qedhere
\end{proof}

\begin{proof}[Proof of Theorem \ref{MainThm1}]
We already completed the proof of Theorem \ref{th:psv}.
Theorem \ref{MainHomotopy} is proved in \cite{GPan}.
Clearly, Theorem \ref{MainThm1} is a direct consequence
of Theorems \ref{MainHomotopy} and \ref{th:psv}.
Thus, Theorem \ref{MainThm1} is proved. \qedhere
\end{proof}

\end{document}